\documentclass[11pt]{article}

\usepackage[T1]{fontenc}
\usepackage[utf8]{inputenc}
\usepackage[english]{babel}
\usepackage{amsmath}
\usepackage{amsthm}
\usepackage{amssymb}
\usepackage{amsfonts}
\usepackage{amsxtra}
\usepackage{enumerate}
\usepackage{verbatim}
\usepackage{color}
\usepackage{a4wide}
\usepackage[mathscr]{eucal}
\usepackage{epsfig,epic,eepic}

\newcommand{\N}{\mathbb N}
\newcommand{\Z}{\mathbb Z}
\newcommand{\Q}{\mathbb Q}
\newcommand{\R}{\mathbb R}

\newtheorem{theoremloc}{Theorem}[section]

\newtheorem{propositionloc}[theoremloc]{Proposition}
\newtheorem{lemmaloc}[theoremloc]{Lemma}
\newtheorem{corollaryloc}[theoremloc]{Corollary}
\newtheorem{factloc}[theoremloc]{Fact}

\newtheorem*{theoremlocstar1}{Theorem 2.7}
\newtheorem*{theoremlocstar2}{Theorem 6.1}

\newtheorem*{remarklocstar}{Remark}

\DeclareMathOperator*{\Time}{\mathcal{T}}
\DeclareMathOperator*{\Space}{\mathcal{S}}

\newcommand{\F}{\mathcal F}

\newcommand{\sss}{S^1}
\newcommand{\ie}{i.e.\ }

\title{Pseudo-Riemannian Geodesic Foliations by Circles}

\author{Pierre Mounoud  \\
	Univ. Bordeaux, IMB, UMR 5251\\ 
	F-33400 Talence, France  \\
	{\it email: pierre.mounoud@math.u-bordeaux1.fr}  \\
	\and 
	Stefan Suhr  \\
	Fachbereich Mathematik, Universit\"at Hamburg\\ 
	Bundesstra\ss e 55\\ 
	20146 Hamburg, Germany \\
	{\it email: stefan.suhr@math.uni-hamburg.de} \\
	}

\date{}

\begin{document}

%

\maketitle

\begin{abstract}
We investigate under which assumptions an orientable pseudo-Riemannian geodesic foliations by circles is generated by an $S^1$-action. We construct 
examples showing that, contrary to the Riemannian case, it is not always true. However, we prove that such an action always exists when the foliation does not 
contain lightlike leaves, i.e. a pseudo-Riemannian Wadsley's  Theorem. As an application, we show that every Lorentzian surface all of whose spacelike/timelike 
geodesics are closed, is finitely covered by $S^1\times \R$. It follows that every Lorentzian surface contains a non-closed geodesic.
\end{abstract}

\section{Introduction}\label{last}

Wadsley's   Theorem states that a foliation by circles is a foliation by closed geodesics of some Riemannian metric if and only if the circles have locally 
bounded length. Besides its significance in its own right, Wadsley's   Theorem is a cornerstone in the study of Riemannian manifolds all of whose geodesics 
are closed. See \cite{besse1} for an overview of the subject.

The year following Wadsley's article (\cite{Wadsley}), Sullivan  gave in \cite{Sullivan} the first examples of compact manifolds endowed with a foliation by circles of 
unbounded length. Sullivan's article contains also a modification of his construction by Thurston. Thurston's example has the additional properties of being 
real-analytic and explicitly given.

Wadsley's Theorem implies that there are no Riemannian metrics making Sullivan's or Thurston's examples geodesic. Of course, there exists connections 
making them geodesic (it is true for any foliation). But are there interesting ones? For example, we could wonder, as  Epstein in \cite{besse1}, if there is a flat 
connection making any of them geodesic.

In this paper, we will focus on pseudo-Riemannian connections. First we prove an  analog of Wadsley's  Theorem for geodesic foliations that do not 
contain any lightlike leaves (see  Theorem \ref{T1.1}). Hence, a geodesic foliation by circles of unbounded length has to have lightlike leaves. With the help 
of Thurston's example, we prove:

\begin{theoremloc} 
There exists a smooth foliation by circles of unbounded length, denoted $\F$, on a smooth compact  pseudo-Riemannian manifold $(M,g)$ such that the leaves of $\F$ are geodesics of $g$. The foliation $\F$ and the metric  $g$ can be chosen so that all the leaves of $\F$ are lightlike or  so that there exist 
leaves of $\F$ of any type. Moreover, in the lightlike case $\F$ and $g$ can be chosen to be real-analytic.
\end{theoremloc} 

The type changing situation is perhaps the more interesting. To begin with, no example of type changing geodesic foliations by circles was known to us. Moreover, this case 
is related to  the existence of pseudo-Riemannian manifolds all of whose geodesics are closed. Indeed, if $(M,g)$ is a pseudo-Riemannian manifold, then there 
exists  a pseudo-Riemannian metric $\bar g$ (called the Sasaki metric) on its  tangent bundle $TM$ such that the orbits of the  geodesic flow of $(M,g)$ are geodesics of $\bar g$ (see 
 Proposition \ref{P2.1}). If we were able to find a non-Riemannian manifold $(M,g)$ so that all its geodesics are closed, then $(TM,\bar g)$ would have a type changing 
geodesic foliation by circles.

Furthermore, it is interesting to note that type changing geodesic foliations satisfy what could be called an ``anti-Wadsley's  Theorem''. By this we mean that a type changing 
geodesic foliation is never generated by an action of $\sss$ (see Proposition \ref{period}). Applying this result to geodesic flows gives:

\begin{theoremlocstar1}
A pseudo-Riemannian  manifold having a geodesic flow that can be periodically reparametrized is  Riemannian or anti-Riemannian.
\end{theoremlocstar1}
Hence, if there exists a (non-Riemannian) pseudo-Riemannian manifold all of whose geodesics are closed it has to be quite complicated. 

In opposition, there exists a well-known family of pseudo-Riemannian manifolds all of whose spacelike geodesics are closed. It is the  family of
pseudo-spheres $(r>0)$:
$$S^n_\nu(r):=\left\{x\in\R^{n+1}\big|\; \langle x,x \rangle_\nu= -\sum_{i=1}^\nu (x^i)^2 +\sum_{i=\nu+1}^{n+1} (x^i)^2 =r^2\right\}$$
It is elementary to see that all spacelike geodesics of $S^n_\nu(r)$ are closed. By inverting all signs these examples can be modified to yield examples of 
pseudo-Riemannian manifolds all of whose timelike geodesics are closed. Note that $S^n_\nu(r)$ is diffeomorphic to $\R^\nu\times S^{n-\nu}$, where $S^{n-\nu}$ is the 
(Euclidean) unit sphere of dimension $n-\nu$. There also exists a classical family of pseudo-Riemannian metrics all of whose lightlike geodesics are closed. This is the  family of
metrics in the conformal class of $S^q\times S^p$ endowed with the metric $g_{S^q}-g_{S^p}$ (when $p=1$ it is called the $q+1$-Einstein's universe). 

Our main result on this problem  is the following (compare with \cite{GroGro}):
\begin{theoremlocstar2}
If $(M,g)$ is a  Lorentzian $2$-manifold all of whose timelike (respectively spacelike) geodesics are closed, then $(M,g)$ is finitely covered by a timelike-SC-metric 
(resp. spacelike SC-metric)  on $S^1\times\R$. 
\end{theoremlocstar2}

It follows that every Lorentzian surface contains non-closed geodesics (Theorem \ref{nozollsurface}). Finally we extend a theorem of 
Guillemin (see \cite{guill1} p.4) concerning compact surfaces, stating that, up to a finite cover, any Lorentzian surface all of whose lightlike geodesics are closed
is in the conformal class of the $2$-dimensional Einstein universe.

\subsection*{Notation}
Throughout the entire text we will assume that $M$ is a connected smooth (i.e. $C^\infty$) manifold. Further ``smooth'' is synonymous for $C^\infty$.

\section{Wadsley's Theorem for pseudo-Riemannian Manifolds}\label{sec1}

Let $\F$ be a smooth foliation by circles of $M$. We can give the following generalization of Wadsley's  Theorem as formulated in \cite{besse1}:

\begin{theoremloc}\label{T1.1}
The following conditions are equivalent:
\begin{itemize}
\item[(1)] There is a smooth pseudo-Riemannian metric making $\F$ a geodesic foliation by non-degenerate geodesics of the same causal character, i.e. 
$g(\dot\gamma,\dot\gamma)>0$ or $g(\dot\gamma,\dot\gamma)<0$ for all geodesics $\gamma$ parameterizing leaves of $\F$. 
\item[(2)] For any compact subset $K$ of $M$, the circles meeting $K$ have bounded length with respect to some (hence every) Riemannian metric.
\item[(3)] Let $\widetilde{M}$ be the double cover of $M$ obtained by taking the two different possible local orientations of the leaves. There is a smooth action 
of the orthogonal group O$(2)$ on $\widetilde{M}$ and the non-trivial deck transformation $\sigma\colon \widetilde{M}\to\widetilde{M}$ is an element of the 
non-trivial component of O$(2)$. Each orbit under the O$(2)$-action consists of two components and each component is mapped diffeomorphically onto a leaf of  $\mathcal{F}$ by the covering projection.
\end{itemize}
\end{theoremloc}

 This formulation of Wadsley's Theorem differs from the formulation in \cite{besse1} only in the replacement of ``Riemannian'' by ``pseudo-Riemannian'' in part
(1) and the condition that all geodesics are nonlightlike. Hence the theorem will follow directly from Wadsley's Theorem as formulated in \cite{besse1} and the 
following proposition. 

\begin{propositionloc}\label{idem}
Let $\mathcal{G}$ be a $1$-foliation of a manifold $M$. Then the following assertion are equivalent:
\begin{itemize}
\item[(1)] There exists a pseudo-Riemannian metric $g$ on $M$ turning $\mathcal{G}$ into a $g$-geodesic foliation by non lightlike geodesics.
\item[(2)] There exists a Riemannian metric $h$ on $M$ turning $\mathcal{G}$ into a $h$-geodesic foliation.
\end{itemize}
\end{propositionloc}

\begin{proof}
(1) $\Rightarrow$ (2): Consider the $g$-orthogonal complement $\mathcal{G}^\perp$ of $T\mathcal{G}$, the bundle of tangent spaces to the leaves of $\mathcal{G}$, and 
the orthogonal projection $P_\mathcal{G}\colon TM\to \mathcal{G}^\perp$. Note that both $\mathcal{G}^\perp$ and $P_\mathcal{G}$ are globally well defined, no 
matter if $\mathcal{G}$ is orientable or not. Choose any Riemannian metric $h_0$ on $M$ and set
$$h:= h_0(P_\mathcal{G}.,P_\mathcal{G}.)+g(\overline{X},.)\otimes g(\overline{X},.)$$
where $\overline{X}$ is any choice of $g$-unit tangent vector in $T\mathcal{G}$, i.e. $g(\overline{X},\overline{X})=\pm 1$. Since we assumed that all leaves of 
$\mathcal{G}$ are nonlightlike geodesics, it follows that $g|_{T\mathcal{G}\times T\mathcal{G}}\neq 0$. Consequently 
$\mathcal{G}^\perp$ is transversal to $T\mathcal{G}$ and $h$ is a Riemannian metric.

Let $U\subseteq M$ be open such that $\mathcal{G}|_U$ is orientable. Choose a $g$-unit vector field $X\in \Gamma(TU)$ tangent to 
$\mathcal{G}$. Since the flowlines of $X$ are $g$-geodesics we have with Koszul's formula
$$0=g(\nabla_X X,Z)=2X(g(X,Z))-Z(g(X,X))+g([X,X],Z)-2g([X,Z],X)=-2g([X,Z],X)$$
if $Z$ is orthogonal to $X$ everywhere.

It is obvious that $X$ is an $h$-unit vector field and the orthogonal complement of $T\mathcal{G}$ w.r.t. $g$ and $h$ coincide. For the Levi-Civita connection 
$\nabla^h$ of $h$ we have 
\begin{align*}
h(\nabla^h_X X,Z)&=2X(h(X,Z))-Z(h(X,X))+h([X,X],Z)-2h([X,Z],X)\\
&=-2h([X,Z],X)=h_0(P_\mathcal{G}[X,Z],P_\mathcal{G} X)+g([X,Z],X)g(X,X)=0
\end{align*}
if $Z$ is orthogonal to $X$ everywhere. Since by construction $h(X,X)\equiv 1$ it follows that $\nabla^h_X X\perp X$ w.r.t. $h$. This implies that 
$\nabla^h_X X=0$ and $\mathcal{G}$ is $h$-geodesic. 

(2) $\Rightarrow$ (1): Since we consider Riemannian metrics as pseudo-Riemannian as well, the claim is immediate.
\end{proof}

 We will need the following refined version of one implication in Wadsley's Theorem. It follows from Proposition \ref{idem} and Theorem A.26 p.\@ 220 from 
 \cite{besse1}.
\begin{theoremloc}\label{T2}
Let $(M,g)$ be a pseudo-Riemannian manifold with a smooth unit (i.e. $g(X,X)=\pm 1$) vector field $X$ such that every flowline is a geodesic circle. Then 
$X$ is the derivative of a smooth circle action on $M$ (i.\@e.\@ its flow is periodic). 
\end{theoremloc}

As we said in the introduction, not only the leaves of a type changing geodesic foliation by circles can have unbounded length but they have to. 
It is actually a  consequence of the results of Wadsley (\cite{Wadsley}) and  Theorem \ref{T1.1}.

\begin{propositionloc}\label{period}
Let $\F$ be an oriented $1$-dimensional geodesic foliation on a  pseudo-Riemannian manifold $(M,g)$. If the leaves of $\F$ are circles with locally bounded length then they all have the same type.
\end{propositionloc}

\begin{proof} 
For any $x\in M$ we denote by $F_x$ the leaf of $\F$ containing $x$. According to Theorem \ref{T1.1}, there exists a vector field 
$Z$ tangent  to $\F$  such that the flow of $Z$ is $2\pi$-periodic (\ie  the leaves of $\F$ are the orbits of an action of $\sss$).
We endow $M$ with an auxiliary Riemannian metric $h$, such that $h(Z,Z)=1$. Hence, for any $x\in M$, the length of $F_x$ divides $2\pi$.

Let $U=\{x\in M ,\, g(Z,Z)\neq 0\}$, it is an open \emph{saturated} (\ie a union of leaves of $\F$) subset of $M$. We can suppose that $U$ is not empty (otherwise 
there is nothing to prove). Let $Z_0$ be the vector field on $U$ defined by $Z_0=1/\sqrt{|g(Z,Z)|} \,Z$. Let $U_0$ be a connected component of $U$. According to  
 Theorem \ref{T2}, there exists $T>0$ such that $\phi$, the flow of $Z_0$,  is $T$-periodic on $U_0$. We define a function $\nu$ on $U_0$ by 
$$\nu(x)=\int_0^T \sqrt{h_{\phi_t(x)}(Z_0,Z_0)}dt.$$ 
For any $x\in U_0$, $\nu(x)$ is a multiple of the length of $F_x$, therefore there exist $n(x)\in \Q$ such that $\nu(x)=2\pi n(x)$. The function $\nu$ being 
continuous it follows that $n(x)$ is  constant on $U_0$ (actually it follows from  Theorem \ref{T2} that if $2\pi$ and $T$ are the least 
periods of the flows of $Z$ and 
$Z_0$ then $n(x)=1$). But when $x$ tends to a point of $\partial U_0$, $g(Z,Z)$ tends to $0$ uniformly on $F_x$ and therefore $\nu(x)$ tends to $+\infty$. Hence 
$\partial U_0=\emptyset$ \ie $U=M$ and $\F$ does not change type.
\end{proof}

Hence, the foliation constructed in section \ref{typechanging} is probably not far from being the simplest type changing geodesic foliation by circles on a compact 
manifold. However, it is perhaps possible to find a $4$-dimensional example using the example of Epstein and Vogt (see \cite{EV}). For lower dimensions, we 
have:
\begin{corollaryloc}
If $(M,g)$ is a $2$ or $3$ dimensional compact pseudo-Riemannian manifold and if $\F$ is a geodesic foliation by circles then the leaves of $\F$ do not 
change type.
\end{corollaryloc}
\begin{proof} 
It is an immediate consequence of  Proposition \ref{period} and Epstein's Theorem (see \cite{Epstein}) saying that the leaves of a foliation by circles 
on a compact $3$ dimensional manifold have bounded length.
\end{proof}

We have the following ``signature-rigidity''  Theorem:

\begin{theoremloc}\label{complique}
A pseudo-Riemannian  manifold $(M,g)$ having a geodesic flow that can be periodically reparametrized is  Riemannian or anti-Riemannian.
\end{theoremloc}

\begin{propositionloc}\label{P2.1}
Let $(M,g)$ be a pseudo-Riemannian manifold. Then there exists a pseudo-Riemannian metric $\overline{g}$ on $TM$ such that the tangent curve of every 
geodesic of $(M,g)$ is a geodesic of $(TM,\overline{g})$ of the same causal character.
\end{propositionloc}

\begin{proof}
The Levi-Civita connection of $(M,g)$ induces a (fiber-wise linear) connection map $K_g\colon TTM\to TM$ with image transversal to the fibers of the bundle 
$\pi_{TM}\colon TM\to M$. Therefore every tangent space of $TM$ is isomorphic to the product 
$\ker (\pi_{TM})_{\ast} \oplus \ker K_g$. Further $\ker (\pi_{TM})_{\ast}$ as well as $\ker K_g$ is isomorphic to a tangent space $TM_p$. Define the metric 
$\overline{g}$ such that the map $(\pi_{TM})_{\ast}\times K_g \colon T(TM_p)_v\to TM_p\times TM_p$ becomes an isometry where $TM_p\times TM_p$ is equipped with 
the product metric $g\oplus g$.

Note that $\pi_{TM}\colon (TM,\overline{g})\to (M,g)$ is a pseudo-Riemannian submersion. Now the claim follows similar to  Lemma 9.44 in \cite{besse2} since the 
tangent curves of geodesics are parallel lifts of the geodesics to the tangent bundle. 

Further note that $((\pi_{TM})_\ast\times K_g)(\ddot{\gamma})=(\dot\gamma,0)$, if $\gamma$ is a geodesic. This implies that the causal character of the tangent 
curve is the same as the causal character of the underlying geodesic. 
\end{proof}

\begin{remarklocstar}
$\overline{g}$ is called the Sasaki metric of $TM$.
\end{remarklocstar}

\begin{proof}[Proof of Theorem \ref{complique}]
We denote by $\F$ the foliation on $TM$ that is generated by the geodesic flow of $g$. According to  Proposition \ref{P2.1}, the leaves of $\F$ are geodesics of the Sasaki 
metric $\overline{g}$ and  all causal type appear in $\F$. It follows from Proposition \ref{period} that if $\F$ can be periodically reparameterized (i.~e.\@ if the leaves of $\F$ 
are closed and have locally bounded length)  then it does not change type.  But as it takes all the possible types, it means that $g$ is Riemannian or anti-Riemannian.
 \end{proof}

\section{Thurston-Sullivan Example}\label{TSE}
We start by recalling Thurston's example of a real-analytic foliation by circles of unbounded length on a compact manifold. The presentation we are giving is elementary 
but it has the drawback that the geometric ideas behind the construction are hidden.  For a geometric description the reader is  invited to look at \cite{Sullivan} or 
\cite{besse1} p. 222.

Thurston's flow lives on the manifold $N\times \sss\times \sss$ where $N$ is the quotient of $H$ the $3$-dimensional Heisenberg group, 
$$H=\left\{
\begin{pmatrix}
  1&x&z\\0&1&y\\0&0&1                                                                                              
\end{pmatrix}\!,\, (x,y,z)\in \R^3
\right\},$$ 
by the left action of the lattice $\Gamma$ given by
 $$\Gamma=\left\{
\begin{pmatrix}
1 &a &c\\0&1&b\\0&0&1
\end{pmatrix}\!,\,  (a,b,c)\in \Z^3\right\}$$
and the circles $\sss$ are identified with $\R/2\pi\Z$. We denote by $p$ the projection from $H\times \R^2$ to $N\times \sss\times \sss$. Let $(t,u)$ be the coordinates on 
$\R^2$ (thus we have coordinates $(x,y,z,t,u)$ on $H\times \R^2$). Let $V_1$ and $V_2$ be the vector fields on $H\times \R^2$ given by 
\begin{eqnarray*} 
V_1&=&\cos(t) \partial_x+\sin(t) (\partial_y+x\partial_z),\\
V_2&=&-\sin(t) \partial_x + \cos(t) (\partial_y+x\partial_z).
\end{eqnarray*}
We note that the vector fields $V_1$, $V_2$,  $\partial_z$, $\partial_t$ and $\partial_u$ are everywhere independent and that they are  invariant by the action of $\Gamma'=
\Gamma\times (2\pi\Z)^2$. Thus, they define a moving frame on the manifold $N\times \sss\times \sss$. 

We will denote by $X$  Thurston's vector field on $H\times \R^2$, it is defined by: 
$$X=\sin(2u)V_1 +2\sin^2(u)\partial_t -\cos^2(u) \partial_z.$$
We see that $X$ is also $\Gamma'$-invariant.

In what follows, we will denote by the same letter a $\Gamma'$-invariant vector field on $H\times \R^2$ and its projection on $N\times \sss\times \sss$.

There is an easy way to see  that the foliation spanned by $X$ on $N\times \sss\times \sss$ has closed leaves of unbounded length. On  $U=\{(x,y,z,t,u)\in H\times \R^2 , 
u\not\equiv 0 \,(\operatorname{mod} \pi)\}$ we define the vector field $W$ by 
$$W=\frac{1}{2\sin^2(u)}X=\frac{1}{\tan(u)}\,V_1 +\partial_t-\frac{1}{2\tan^2 (u)}\,\partial_z.$$
We can easily compute $\phi$, the flow of $W$:
$$\begin{array}{ll}
\phi_s{
(x,y,z,t,u)}=&\!\!\!\left(\frac{\sin(t+s)-\sin(t)}{\tan(u)}
+x,\frac{\cos(t)-\cos(t+s)}{\tan(u)} +y, \right. 
\\
& 
\left. z+\frac{\sin(2t)-\sin(2t+2s)}{4\tan^2(u)}
+\frac{(\cos(t+s)-\cos(t))(\sin(t)+ x \tan(u))}{\tan^2(u)},
t+s,u\right). 
\end{array}
$$
Hence, it is  clear that the projection of $W$ on $N\times \sss\times \sss$ is a vector field defined on $p(U)$ which has a $2\pi$-periodic flow. It follows that the 
integral curves of $X$ on $p(U)$ are all closed and that their length is unbounded (if we choose a  Riemannian metric $h$ such that $h(X,X)=1$, their length is equal to 
$2\pi/\sin^2(u)$). To conclude, we need to check that the integral curves on $N\times \sss\times \sss\setminus p(U)$ are also closed. It is true because on this closed 
subset (known as the ``bad set'') $X=\partial_z$ and because there is an element of $\Gamma$ acting by translation along the coordinate $z$.

In what follows, we will principally use a smooth deformation of $X$, denoted $X_\xi$, instead of $X$ itself. This vector field is defined by:
$$X_\xi=2\xi(u)\cos(u) V_1 +2\xi^2(u)\partial_t -\cos^2(u) \partial_z,$$
where  $\xi : \R \rightarrow \R$  is  the smooth, $2\pi$-periodic function defined by $\xi(u)=e^{-1/\sin^2(u)}$ if $0<u<\pi$ and by $\xi(u)=-e^{-1/\sin^2(u)}$ if $-\pi<u<0$. 
We remark that $\xi^{(k)}(n\pi)=0$ for any $k\in \N$ and any $n\in \Z$, consequently $|\xi|$ is also smooth.
\begin{figure}[h!]
\centering
\includegraphics[width=0.4\textwidth]{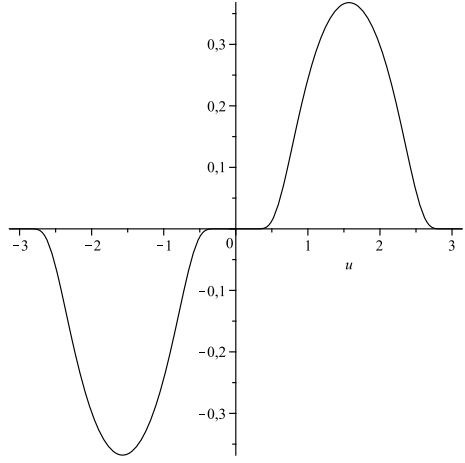}
\caption{The graph of $\xi$}
\end{figure}

As above,  we define $W_\xi$ as $\frac{1}{2\xi^2}X_\xi$. The flow of $W_\xi$ is very similar to the flow of $W$, in particular it is  also periodic, therefore $X_\xi$ also 
generates a foliation by circles with unbounded length. This foliation is no more real-analytic.

\section{The Lightlike Construction}
In this section, we construct a real-analytic  Lorentzian metric $g$ on $N\times \sss\times \sss$ such that $D_XX=0$ (where $D$ is the Levi-Civita connection of $g$) and 
$g(X,X)=0$ everywhere. First, we remark that the vector fields  $X$, $V_1$, $V_2$, $2\partial_t+\partial_z$ and $\partial_u$ are everywhere independent and therefore they 
define a moving frame on $N\times \sss\times \sss$. Then, we define $g$ as the metric which is given in the moving frame $(X, \partial_u,V_1, V_2, 2\partial_t+\partial_z)$ by 
the  matrix 
$$\begin{pmatrix}
0&1&0&0&0\\
1&0&0&0&0\\
0&0&1&0&0\\
0&0&0&1&0\\
0&0&0&0&1
  \end{pmatrix}.
$$
This is clearly a real-analytic Lorentzian metric. The $1$-form $X^\flat$ (defined by $X^\flat(v)=g(X,v)$) is $du$ and therefore is closed. For any vector field $Z$, we have:
\begin{eqnarray*}
0=dX^\flat(X,Z)&=&X\cdot g(X,Z)-W\cdot g(X,X)-g(X,[X,Z])\\&=&g(D_XX,Z)+g(D_ZX,X)=g(D_XX,Z)+\frac{1}{2}Z\cdot g(X,X)\\&=&g(D_XX,Z).
\end{eqnarray*}
Hence $D_XX=0$. 

Let us remark that the Riemannian metric on the distribution spanned by $\{V_1, V_2, 2\partial_t+\partial_z\}$ can be replace by any pseudo-Riemannian metric. As there is 
so many metrics making $X$ geodesic, we may wonder if there is a best one (as for example a flat one).

\section{The Type Changing Construction}\label{typechanging}

In this section, we will provide a metric $g$ such that $D_{X_\xi}{X_\xi}=0$ and such that $g(X_\xi,X_\xi)$ changes signs. We start the construction in a neighborhood of the 
bad set \ie for $u$ close to $0 \,(\operatorname{mod} \pi)$. As above  we use a moving frame to write the metric. Let $g_0$ be the symmetric $2$-tensor field whose 
expression in the moving frame $(\partial_t, V_1,V_2,\partial_z,\partial_u)$ is given by the following matrix (denoted $G_0$):
$${\renewcommand{\arraystretch}{1.3}\begin{pmatrix}\renewcommand{\arraystretch}{1.5}

\frac{1}{4}\cos^4(u)                     &\frac{|\xi(u)|}{2 \cos(u)} & 0     & \frac{1}{2}\cos^2(u)\xi^2(u)-\frac{\xi(u)|\xi(u)|}{\cos^2(u)} & 0 \\
\frac{|\xi(u)|}{2 \cos(u)}               &   0                & \frac{1}{2} &          \frac{ |\xi(u)|\xi^2(u)}{\cos^{3}(u)}   & 0  \\
        0                                &      \frac{1}{2}   & 0           &             \frac{\xi(u)}{\cos(u)}        & 0  \\
\frac{1}{2}\cos^2(u)\xi^2(u)-\frac{\xi(u)|\xi(u)|}{\cos^2(u)} &  \frac{|\xi(u)|\xi^2(u)}{\cos^{3}(u)} &  \frac{\xi(u)}{\cos(u)}  &  \xi^4(u)        & -1  \\
                      0                  &         0          &    0         &               -1                      & 1
  \end{pmatrix}}$$
For $u=0\,(\operatorname{mod} \pi)$, it is easy to verify that  this matrix defines a symmetric bilinear form of signature $(3,2)$ (in particular non degenerate). It follows by 
continuity that there exists $\eta>0$ such that $g_0$ is a well defined pseudo-Riemannian metric of signature $(3,2)$ on $N\times \sss \times \,]-\eta,\eta[$ and on  
$N\times \sss \times \,]\pi-\eta,\pi+\eta[$.

Applying $G_0$ to $(2\xi^2(u),2\xi(u)\cos(u),0,-\cos^2(u), 0)$, we get 
$$(2\xi(u)|\xi(u)|,0,0,0,\cos^2(u)).$$
It implies that  $g(X_\xi,X_\xi)=4\xi^3(u)|\xi(u)|$ (and therefore $g(W_\xi,W_\xi)=\frac{|\xi(u)|}{\xi(u)}=\pm 1$) and that 
$$X_\xi^\perp=\operatorname{span}\big\{V_1,V_2,\partial_z, -\cos^2(u)\partial_t + 2\xi(u)|\xi(u)|\partial_u\big\}.$$

At this time, the metric $g_0$ is only defined on a neighborhood of the bad set. In order to extend it, we write it in the moving frame 
$(W,V_1,V_2,\partial_z,-\cos^2(u)\partial_t + 2\xi(u)|\xi(u)|\partial_u)$. It reads:
$$\begin{pmatrix}
\frac{|\xi(u)|}{\xi(u)}& 0   \\0&L(u)
  \end{pmatrix},$$
where $L(u)$ is the following matrix:
$${\renewcommand{\arraystretch}{1.3}\begin{pmatrix}
  0                           & \frac{1}{2}   &\frac{|\xi^3(u)|}{\cos^3(u)}       & \frac{-\cos(u)}{2}|\xi(u)|     \\
   \frac{1}{2}              & 0                 &  \frac{\xi(u)}{\cos(u)}                     & 0                                \\ 
\frac{|\xi^3(u)|}{\cos^3(u)} &  \frac{\xi(u)}{\cos(u)}   &             \xi^4(u)                   & \frac{-\cos^4(u)}{2}\xi^2(u)-2 \xi(u)|\xi(u)|  \\
\frac{-\cos(u)}{2}|\xi(u)|   & 0             &   \frac{-\cos^4(u)}{2}\xi^2(u)-2 \xi(u)|\xi(u)|    &   \frac{1}{4}\cos^8(u)+4\xi^4(u)
 \end{pmatrix}}.
$$

When $u\in \,]0,\eta[\,\cup\,]\pi-\eta,\pi[$, the matrix $L(u)$ lies in the space on bilinear forms of signature $(2,2)$ \ie in $\operatorname{SL}(4,\R)/\operatorname{SO}(2,2)$.
As this space is a connected manifold there exists a smooth map $M:\,]0,\pi[\,\rightarrow \operatorname{SL}(4,\R)/\operatorname{SO}(2,2)$ such that $M$ and $L$ coincide 
on $]0,\eta/2[\,\cup\,]\pi-\eta/2,\pi[$. Let $g_1$ be the metric on $N\times \sss\times\, ]0,\pi[$ whose matrix in the moving frame $(W,V_1,V_2,\partial_z,-\cos^2(u)\partial_t + 
2\xi(u)|\xi(u)|\partial_u)$ is $\begin{pmatrix} 1&0\\0&M(u) \end{pmatrix}$.

When  $u\in\, ]-\pi,-\pi+\eta[\,\cup\,]-\eta,0[$,  the matrix $L(u)$ lies in the space on bilinear forms of signature $(3,1)$ \ie in $\operatorname{SL}(4,\R)/\operatorname{SO}(3,1)$. 
As $\operatorname{SL}(4,\R)/\operatorname{SO}(3,1)$ is a connected manifold there exists a smooth map $N:\,]-\pi,0[\,\rightarrow 
\operatorname{SL}(4,\R)/\operatorname{SO}(3,1)$ such that $N$ and $L$ coincide on $]-\pi,-\pi+\eta/2[\,\cup\,]-\eta/2,0]$. Let $g_2$ be the metric on $N\times \sss\times\, 
]-\pi,0[$ whose matrix in the moving frame $(W,V_1,V_2,\partial_z,-\cos^2(u)\partial_t + 2\xi(u)|\xi(u)|\partial_u)$ is $\begin{pmatrix} -1&0\\0&N(u) \end{pmatrix}$.

We can now give the metric, we are looking for. By construction the metrics $g_0$, $g_1$ and $g_2$ glue together in a smooth metric $g$ on $M\times \sss\times\sss$. 
Moreover on $p(U)=N\times \sss\times (]-\pi,0[\,\cup\,]0,\pi[)$, we have:
\begin{align*} g(W_\xi,W_\xi)&=\pm 1\\
X_\xi^\perp&=\operatorname{span}\{V_1,V_2,\partial_z, -\cos^2(u)\,\partial_t + 2\xi(u)|\xi(u)|\,\partial_u\}
\end{align*}

We have to prove now that  the foliation generated by $X_\xi$ is geodesic.  It will be done with the help of the following facts:
\begin{factloc}\label{density}
$D_{X_\xi}X_\xi=0$ on $N\times \sss\times \sss$ if and only if $D_{W_\xi}W_\xi=0$ on $p(U)$.
\end{factloc}

\begin{proof} As $X_\xi=2\xi^2(u)W$ and $X_\xi.\xi^2(u)=0$, we have $D_{X_\xi}X_\xi=4\xi^4(u) D_{W_\xi}W_\xi$  therefore the equivalence is clear on $p(U)$. But $p(U)$ is 
dense in $N\times \sss\times \sss$ and $D_{X_\xi}X_\xi=0$ on $p(U)$ is also equivalent to  $D_{X_\xi}X_\xi=0$ on the whole manifold.
\end{proof}

Koszul's formula entails easily the next fact, the proof of which is left to the reader.
\begin{factloc}\label{Koszul}
Let $Z$ be a  vector field on a pseudo-Riemannian manifold $(M,g)$ such that $g(Z,Z)=\pm 1$. We have $D_ZZ=0$ if and only if $Z$ preserves its orthogonal distribution.
\end{factloc}
As $W_\xi^\perp$ is spanned by $V_1,V_2,\partial_z, -\cos^2(u)\,\partial_t + 2\xi(u)|\xi(u)|\,\partial_u$, facts \ref{density} and \ref{Koszul} imply:
\begin{factloc}\label{crochets}
 $D_{X_\xi}X_\xi=0$ if and only if 
\begin{align*} 
g(W_\xi,[W_\xi,V_1])&=g(W_\xi,[W_\xi,V_2])=g(W_\xi,[W_\xi,\partial_z])\\
&=g(W_\xi,[W_\xi, -\cos^2(u)\,\partial_t + 2\xi(u)|\xi(u)|\, \partial_u])=0.
\end{align*}
\end{factloc}
Computing the brackets above, we find:
\begin{eqnarray*}
\, [W_\xi,V_1] &= & V_2, \\
\, [W_\xi, V_2 ]&=&\frac{\cos(u)}{\xi(u)}\, \partial_z-V_1,\\
\, [W_\xi,\partial_z] &=&0, \\
\, [W_\xi, Y ]&=& -2\xi(u)|\xi(u)|\,\partial_u\cdot\frac{\cos(u)}{\xi(u)}\,V_1 +\frac{\cos^3(u)}{\xi(u)}\,V_2 +\xi(u)|\xi(u)|\,\partial_u\cdot\frac{\cos^2(u)}{\xi^2(u)}\,\partial_z, 
\end{eqnarray*}
where $Y=-\cos^2(u)\,\partial_t + 2\xi(u)|\xi(u)|\,\partial_u$. All these brackets are tangent to $W_\xi^\perp$ therefore, by fact \ref{crochets}, $D_{X_\xi}X_\xi=0$.

It is not difficult to modify the metric $g$ in order to have $D_XX=0$ (\ie to make Thurston's flow geodesic and not a smooth modification of it). At the beginning of the 
construction, we just have to replace the function $\xi$ by the function sinus and to remove the absolute values. The metric becomes real-analytic in the neighborhood of the 
bad set but is  only smooth globally (the gluing process is unchanged). Unfortunately,  in that case we have $g(X,X)\geq 0$, the vector field  $X$ is nowhere timelike.

To conclude this section, we prove that $X_\xi$ preserves the volume form of $g$ \ie that $\operatorname{div}X_\xi=0$. We do it  because the geodesic flow of any 
pseudo-Riemannian manifold is divergence free. 
\begin{propositionloc} 
The vectorfield $X_\xi$ is divergence free.
\end{propositionloc}

\begin{proof} We denote the volume form associated to  $g$ by $\omega_g$. We want to compute $\mathcal L_{X_\xi}\omega_g=d (i_{X_\xi} \omega_g)$.
We see that, if $S$ and $T$ are two vector fields chosen among $\partial_t$, $V_1$, $V_2$, $\partial_z$ and $\partial_u$, then $[S,T]$ belongs to 
$\operatorname{span}\{V_1,V_2,\partial_z\}$. It means that all the ''brackets terms'' in the computation of $d(i_{X_\xi}\omega)\,(\partial_t,V_1,V_2,\partial_z,\partial_u)$ 
vanish. Moreover, the function $\omega_g(\partial_t,V_1,V_2,\partial_z,\partial_u)$ is clearly constant on the submanifolds $N\times \sss\times\{u\}$. Therefore:
$$d(i_{X_\xi}\omega)\,(\partial_t,V_1,V_2,\partial_z,\partial_u)=\partial_u\cdot \omega_g(X_\xi,\partial_t,V_1,V_2,\partial_z)=0. $$
\end{proof}

\section{The $2$-Dimensional Case}\label{sec2}

Adapting the terminology of \cite{besse1}, we will say that a pseudo-Riemannian metric is a timelike-SC-metric (respectively spacelike-SC-metric) if all its timelike 
(resp. spacelike) geodesic are simply closed geodesic with the same (pseudo-Riemannian) length.

\begin{theoremloc}\label{T2.1}
If $(M,g)$ is a  Lorentzian $2$-manifold all of whose timelike (respectively spacelike) geodesics are closed, then $(M,g)$ is finitely covered by a timelike-SC-metric 
(resp. spacelike SC-metric)  on $S^1\times\R$. 
\end{theoremloc}

\begin{remarklocstar}
(i) The assumption on the topology of $M$ cannot be removed as  $S^2_1$ is diffeomorphic to $\R\times S^1$ and  
all spacelike 
geodesics are  simply closed. Further $S^2_1$ admits non-orientable quotients. So $M$ can be diffeomorphic to the M\"obius strip. 

(ii) The quotient of the pseudo-sphere $S^2_1$ by a finite order subgroup of its group of direct isometries  gives an example of a metric on $S^1\times \R$ all of whose spacelike geodesics are closed. However, when the subgroup is not trivial this metric is not a spacelike-SC-metric.

(iii) It is not possible to extend theorem \ref{T2.1} to a density result for non-closed geodesics, since one can construct pseudo-Riemannian metrics on any surface $M$ with with an open set of directions $U\subseteq TM$ such that any geodesic with tangents in $U$ are closed. 
\end{remarklocstar}

We split the proof of  Theorem \ref{T2.1} into two cases. One proving the closed case and the other proving the non-compact case.

\begin{proof}[Proof of theorem \ref{T2.1} for the non-compact case]
We start the proof by giving two lemmas.

\begin{lemmaloc}\label{L2.0}
Let $(M,g)$ be a Lorentzian $2$-manifold and $\gamma,\eta\colon S^1\to M$ homotopic, intersecting, smooth, regular  and spacelike/lightlike/timelike curves. 
Then $\gamma$ and $\eta$ have the same causal type.
\end{lemmaloc}

\begin{proof}
We prove the claim by contradiction. Assume that there exist two closed, homotopic, intersecting, smooth, regular and spacelike/lightlike/timelike curves 
$\gamma$ resp. $\eta$ with different causal types. 

Let  $\widetilde \eta$ and $\widetilde \gamma$ be lifts of these curves to $\widetilde M$ (the universal cover of $M$) with a non empty intersection. 
First note that $\widetilde{\gamma}$ and $\widetilde{\eta}$ do not self intersect. This can easily be seen from the theorem of Poincar\'e-Bendixson. 
We can therefore extend the tangents of both $\widetilde{\gamma}$ and $\widetilde{\eta}$ to non singular vector fields of the respective causal type. 
Orienting $T\widetilde{M}$ at one intersection of $\widetilde{\gamma}$ and $\widetilde{\eta}$ such that $\{\dot{\widetilde{\gamma}},\dot{\widetilde{\eta}}\}$ 
forms a positively oriented basis, induces, via the vector fields, an orientation on all of $\widetilde{M}$.

As the loops $\gamma$ and $\eta$  are homotopic  $\widetilde \eta$ and $\widetilde \gamma$ form a bigon $D$ (i.~e.\@ $\widetilde \eta$ and 
$\widetilde \gamma$ intersect at least twice). 
Now it is well known that at consecutive intersections the tangents of $\widetilde{\gamma}$ and $\widetilde{\eta}$ induce different orientations. This is clearly a 
contradiction. 
\end{proof}

Our problem being invariant by finite covering, we only need to consider oriented and time-oriented Lorentzian $2$-manifolds.

\begin{lemmaloc}\label{L2.1}
Let $(M,g)$ be a pseudo-Riemannian manifold with a foliation by non-lightlike closed geodesics. If the foliation is orientable, the fundamental class of any closed 
geodesic has a multiple that lies in the center of $\pi_1(M)$. 
\end{lemmaloc}
\begin{proof}[Proof of lemma \ref{L2.1}]
Let $X$ be a vectorfield tangent to the foliation and such that $|g(X,X)|=1$. We know by Theorem \ref{T2} that under our assumption the flow of $X$ exists and is 
$T$-periodic for some $T>0$, hence it induces a free $\R/T\Z$-action on $M$. We denote this flow by $\phi$. Hence, for all $x\in M$ the curves $\phi_{[0,T]}(x)$ 
are freely homotopic as closed curves. Denote the fundamental class of those geodesics with $\sigma$.

Now let $\tau\in\pi_1(M)$ be arbitrary. Choose any representative $\eta$ of $\tau$ and define the map $F\colon S^1\times  \R/T\Z\to M$, 
$(s,t) \mapsto \phi_t(\eta(s))$. By construction we have $F_\ast(e_1)=\tau$ and $F_\ast(e_2)=\sigma$, where $e_1$ generates $\pi_1(S^1\times\{1\})$ and $e_2$ 
generates $\pi_1(\{1\}\times  \R/T\Z)$, so that 
$$\sigma\ast\tau=F_\ast(e_2)\ast F_\ast(e_1)=F_\ast(e_2\ast e_1)=F_\ast(e_1\ast e_2)=F_\ast(e_1)\ast F_\ast(e_2)=\tau \ast\sigma.$$
This shows that $\sigma$ commutes with $\tau$. Since $\tau$ was arbitrary, we see that $\sigma$ lies in the center of $\pi_1(M)$. 
\end{proof}

 We assume that all spacelike geodesics are closed, since the other case follows by considering $(M,-g)$ instead of $(M,g)$. Note that if $M$ is oriented and $(M,g)$ is time-oriented, the set of timelike vectors $\Time$ 
and the set of spacelike vectors $\Space$ are fiber bundles isomorphic 
to a double copy of the trivial fiber bundle $M\times \R^2$. Therefore the projections induce isomorphisms between $\pi_1(M)$ and the fundamental groups of 
each connected component of $\Time$ and $\Space$. By our assumption and  Proposition \ref{P2.1} we know that each component of  $\Space$ is 
foliated by closed non-lightlike geodesics, so we can apply  Lemma \ref{L2.1}. Thus the fundamental class of each closed geodesic has a multiple that must be in the center of $\pi_1(M)$. 

It is well known (\cite{stillwell}) that the fundamental group of a non-compact surface is free. Further the center of a free group is trivial iff the group is generated by 
a single element. The fundamental group of a non-compact surface is generated by a single element iff the surface is diffeomorphic to either $\R^2$ or $S^1\times \R$. 

It is further well known that no spacelike curve in a Lorentzian $2$-manifold is contractible. Therefore the fundamental class of the closed geodesics cannot be trivial. The only possibility left is $\pi_1(M)\cong \Z$, which implies $M\cong S^1\times \R$.

Let $(\widetilde M, \widetilde g)$ be the universal cover of $(M,g)$. Let $f$ be a generator of the fundamental  group of $M$ seen as a group of isometries of  $(\widetilde M, \widetilde g)$. According to Theorem \ref{T2} there exist $k\in \Z$ and $T>0$ such that for any unit speed spacelike geodesic $\widetilde\gamma$ of $\widetilde M$ and any $t_0\in \R$, we have 
$$\widetilde\gamma(t_0+T)=f^k(\widetilde\gamma(t_0)) \qquad (*)$$ 
($f^k$ corresponds to the homotopy class $\sigma$ defined  in the proof of  Lemma \ref{L2.1}). Hence, the group generated by $f^k$ acts by translation on 
 the spacelike geodesics. 

Let $\widehat M =\widetilde M/\langle f^k\rangle$ and $\widehat g$ be the metric induced by  $\widetilde g$ on $\widehat M$. Let $\widehat \gamma$ be a unit speed spacelike geodesic of $\widehat g$. Its lift to $\widetilde g$ is a unit speed spacelike geodesic of $\widetilde M$.
Let $X$ be a nowhere zero lightlike vector field on $\widetilde M$  which is invariant by $f$. Let $c$ be the integral curve of $X$ containing  $\widetilde \gamma(0)$. As $X$ is invariant by $f$, we know that $f^k(c)$ is an integral curve of $X$. Therefore $c=f^k(c)$ or $c\cap f^k(c)=\emptyset$.
By Lemma \ref{L2.0}, we see that $c\cap f^k(c)=\emptyset$. It follows that $c$ and $f^k(c)$ delimit a fundamental domain $D$ for the action of $\langle f^k\rangle$ 
on $\widetilde M$. Moreover, by recalling that $(\widetilde{M},\widetilde{g})$ is time orientable as well as orientable and considering the angle between 
$\widetilde \gamma'$ and $X$, we see that $\widetilde \gamma$ can intersect only once $c$ or $f^k(c)$. Hence, the relation $(*)$ implies that 
$\widetilde \gamma([0,T])\subset D$.
As a simply connected Lorentzian surface does not contain any closed spacelike curves, the geodesic  $\widetilde \gamma$  does not self-intersect. It 
follows 
that $\widehat \gamma([0,T])$ is a simply closed geodesic of length $T$ and that $(\widehat M,\widehat g)$ is the finite covering we were 
looking for.
\end{proof}

\begin{proof}[Proof of  Theorem \ref{T2.1} for the closed case]
Assume that there exists a closed Lorent\-zian $2$-manifold $(M,g)$ all of whose timelike or spacelike geodesics are closed. We will consider the case of closed  
spacelike geodesics only, since the other case follows by considering $(M,-g)$ instead of $(M,g)$.

By  Theorem \ref{T1.1} and  Proposition \ref{P2.1} the Riemannian arclength of the  spacelike geodesics is locally bounded. Therefore all closed 
spacelike geodesics of the common period are homotopic as loops. Denote this homotopy class with $\sigma$. 

Next note that every smooth non-self-intersecting compact  spacelike curve can be extended to a smooth  spacelike $1$-foliation. 
Furthermore, by imitating the end of the proof of the non compact case, we see that there exists a finite covering $\widehat M$  of $M$  that contains simply 
closed spacelike geodesics. We choose a smooth spacelike 
simply closed geodesic $\eta$ on $\widehat M$. 
Extend $\eta$ to a  spacelike foliation $\F_{space}$ of $\widehat M$.

Now it is easy to see that every spacelike geodesic  $\gamma$ intersecting $\eta$ is tangent at one point to $\F_{space}$. Lift  $\gamma$ and $\F_{space}$ to the universal 
cover.  Denote the lifts with $\widetilde\gamma$, $\widetilde{\eta}$ and $\widetilde{\F}_{space}$. Consider the 
function along $\widetilde\gamma$ that associate to  every point the Lorentzian angle between $\dot {\widetilde \gamma}(t)$ and $X(\widetilde \gamma(t))$ ($X$ 
a given unit vector field tangent to $\widetilde{\F}_{space}$). By construction  $\widetilde\gamma$ and $\widetilde\eta$ 
will intersect at least twice. But since the universal cover is homeomorphic to $\R^2$, the angle at two consecutive intersections has to change sign 
($\widetilde\eta$ disconnects the universal cover). Therefore the angle vanishes somewhere, i.e.\@ $\gamma$ is tangent to $\F_{space}$.

Note that the set of unit tangents $T^1\F_{space}$ to $\F_{space}$ forms a compact subset of the  spacelike vectors in $TM$. Hence, by  Theorem \ref{T1.1}, the set of unit spacelike geodesics intersecting $\F_{space}$ is a compact subset of the tangent bundle of $M$.  
But, we just proved that this set contains all the unit spacelike geodesics that intersect $\eta$ and this set is not contained in a compact subset of $TM$ (it contains the set of unit vectors above any point of $\eta$).


\end{proof}

With  Theorem \ref{T2.1} we can now answer the question of Lorentzian surfaces all of whose geodesics are closed.

\begin{theoremloc}\label{nozollsurface}
Every Lorentzian surface contains a nonclosed spacelike or timelike geodesic. Especially, there does not exist any Lorentzian surface all of whose geodesics are 
closed.
\end{theoremloc}

\begin{proof}
Any Lorentzian surface, all of whose spacelike geodesics are closed, is diffeomorphic to $S^1\times\R$ by  Theorem \ref{T2.1}. Now if there exists a Lorentzian 
surface all of time- and spacelike geodesics are closed, there would exist intersecting closed geodesics in the same fundamental class (Recall that no closed 
time- or spacelike curves is contractible in a Lorentzian surface). Simply consider iterates of closed geodesics if necessary. But this is
impossible by  Lemma \ref{L2.0}. 
\end{proof}

We close these notes with some comments on the case that all lightlike geodesics are closed and the case that all geodesics are closed.  In \cite{guill1} a compact 
pseudo-Riemannian manifolds $(M,g)$ is called ``Zollfrei'' if  the geodesic flow restricted to the set $H:=\{g(v,v)=0, v\neq 0\}$ forms a fibration of $H$ by $S^1$'s. It 
is equivalent to the existence of a free action of $S^1$ on $H$ whose orbits are the integral curves of the geodesic flow. For a Lorentzian surface, being Zollfrei is 
equivalent (up to a finite cover) to having closed lightlike geodesics. The Zollfrei metrics on surfaces are given by the following theorem. This result extends a 
theorem of Guillemin (see \cite{guill1} p.4) where the manifold is supposed to be compact.

\begin{theoremloc}\label{T2.2}
Let $g_{E}$ be the Einstein metric on $S^1\times S^1$, i.e. the metric $g_{S^1}-g_{S^1}$.
Let $(X,g)$ be any Lorentzian surface all of whose lightlike geodesics are closed. Then there exists a covering map $p\colon S^1\times S^1\rightarrow X$ such 
that $p^\ast g$ and $g_E$ are conformally equivalent. In particular $X$ is compact.
\end{theoremloc}

\begin{proof}
Let $(X,g)$ be a  Zollfrei surface. There exists a finite covering $p_0: \widehat X\rightarrow X$ such that $p_0^*g$ has two orientable lightlike linefields. We denote by $\F_1$ and $\F_2$ the foliations generated by those linefields. As $\F_1$ is a transversaly oriented codimension $1$ foliation by circles, 
the map $\pi\colon\widehat X\rightarrow \widehat X/\F_1$ is a smooth fibration. Let $F$ be a leaf of $\F_2$. 
The leaf $F$ is compact therefore $\pi|_{F}$  too is a covering  of its image. It implies that $\pi(F)$ is a circle. Hence, $\widehat X/\F_1$ is also a 
circle and $\widehat X$ has to be a torus.

Thus, $\widehat X$ is a compact foliated bundle with vertical foliation $\F_1$ and with horizontal foliation $\F_2$. 
Otherwise said, there exist a diffeomorphism of $S^1$ of finite order $k$, denoted by $\phi$, and  a diffeomorphism $X\rightarrow (S^1\times \R) / G$,  where $G$ is the group generated by the map $f\colon (x,t)\mapsto (\phi(x),y+1)$, that sends the foliations $\F_{1,2}$  to the horizontal  resp. vertical foliation of  $(S^1\times \R) / G$.
Let $\Gamma$ be the subgroup of $G$ generated by $f^k$, $\widetilde X=(S^1\times \R) / \Gamma$ and $p$ be the covering $\widetilde X\rightarrow \widehat X$.  Clearly, the lifts to $\widetilde X$ of the foliations $\F_i$ form a product. It means that there exists a diffeomorphism $h\colon \widetilde X\rightarrow  S^1\times S^1$ such that $h^*g_E$ and $p^*(p_0^*g)$ are conformal.
\end{proof}

\begin{remarklocstar}
We could combine  Theorem \ref{nozollsurface} and Theorem \ref{T2.2} to the result stating: If all geodesics of one causal type (timelike, spacelike, lightlike) are 
closed, then all geodesics of all other types are non-closed.
\end{remarklocstar}

\end{document}